\documentclass[11pt]{article}

\usepackage{graphicx, url}
\usepackage[margin=0.9in]{geometry}
\usepackage{amsthm, amsmath, amsfonts, amssymb}
\usepackage{mathtools}
\usepackage{dsfont} 
\usepackage{calc}
\usepackage{booktabs, makecell, longtable}
\usepackage{enumerate}
\usepackage{marvosym}
\usepackage{setspace}
\usepackage{color}
\usepackage{array}   
\usepackage{subcaption}
\usepackage{lscape}
\usepackage{pifont}
\usepackage{sectsty}

\usepackage{tikz}
\usetikzlibrary{positioning,fit,calc, arrows,shapes, svg.path, calc}
\tikzset{> =  triangle 45}

\newtheorem{lemma}{Lemma}
\newtheorem{theorem}{Theorem}

\theoremstyle{definition}

\newtheorem{observation}{Observation}

\newcommand{\pfrac}[2]{\left(\frac{#1}{#2}\right)}

\newsavebox{\twosubbox}

\newsavebox{\threesubbox}

\title{Minimum size generating partitions and their application to demand fulfillment optimization problems}
\author{Bo Jones \qquad\qquad John Gunnar Carlsson \\ University of Southern California}

\begin{document}

\maketitle

\begin{abstract}
For $n$ and $k$ integers we introduce the notion of some partition of $n$ being able to \textit{generate} another partition of $n$. We solve the problem of finding the minimum size partition for which the set of partitions this partition can generate contains all size-$k$ partitions of $n$. We describe how this result can be applied to solving a class of combinatorial optimization problems.
\end{abstract}

\begin{section}{Problem}

Unless otherwise specified we assume the convention that a partition $\lambda = (\lambda_1, \lambda_2, \ldots, \lambda_l)$ is such that $\lambda_1 \geq \lambda_2 \geq \cdots \geq \lambda_l.$

Let $n, k$ be integers. We say that a partition $\mu$ of $n$ \textit{\textbf{generates}} a partition $\lambda$ of $n$ if there exists a decomposition of $\mu$ into disjoint subsets $\mu_{S_1}, \ldots \mu_{S_{|\lambda|}}$ such that for each $i \in \{1,\ldots,|\lambda|\}$,  $\mu_{S_i} \vdash \lambda_i$. If for all $\lambda \vdash n$ such that $|\lambda| \leq k$ we have $\mu$ generates $\lambda$ we will say that $\mu$ generates all $k$-partitions of $n$. We wish to find the $\mu$ of minimum size such that $\mu$ generates all $k$-partitions of $n$.

\end{section} 

\begin{section}{Solution}

We can construct an optimal solution to this problem as follows. 
\begin{enumerate}
	\item Let $\mu_1 = \lceil n/k\rceil$. 
	\item For $i \neq 1$ let $\mu_i =  \lceil (n - \sum_{j < i} \mu_j) /k\rceil$.
\end{enumerate}

\end{section} 

\begin{section}{Cost of the solution}

Let $L(n,k)$ be the size, i.e. number of nonzero terms, of the $\mu$ constructed in this way for inputs $n$ and $k$. We do not have a closed form expression for $L(n,k)$, but it is clear that we have the following recurrence relation
\[
	L(n,k) = 1 + L(n - \lceil n/k \rceil,k), \quad L(0,k) = 0,
\]
equivalently
\[
	L(n,k) = 1 + L\left(\left\lfloor \pfrac{k-1}k n \right \rfloor, k\right), \quad L(1,k) = 1.
\]
Let $L'(n,k)$ satisfy the recurrence relation
\[
	L'(n, k) = 1 + L\left(\pfrac {k-1}k n, k\right), \quad L'(1,k) = 1.
\]
Then clearly for any fixed $n,k$, we have $L'(n,k) \geq L(n,k)$. In addition $L'(n,k)$ is increasing in $n$. Thus $L(n,k)$ must have value less than or equal to that of 
\[
	L'\left(\pfrac k {k-1} ^{\lceil \log_{k/(k-1)}(n) \rceil}, k\right) =  \lceil \log_{k/(k-1)}(n) \rceil + 1.
\]

We thus have $|\mu| \leq \lceil \log_{k/(k-1)}(n) \rceil + 1.$ This becomes a terrible upper bound as $k$ increases, but we can quickly compute $|\mu|$ exactly for given $n$ and $k$ by simply running the algorithm. Table \ref{tab:sequences} gives the values of $|\mu|$ for $n \in \{1,\ldots, 20\}$, $k \in \{1,\ldots,10\}$.

\end{section}

\begin{section}{Solution is feasible}

We will make use of the following two key observations.

\begin{observation}[pigeonhole principle]
\label{obs:average}
For every $\lambda \vdash n$ with $|\lambda| \leq k$ there exists some $i$ such that $\lambda_i \geq \lceil n/ k \rceil$.
\end{observation}

\begin{observation}
\label{obs:completion}
Let $\{\lambda_1, \ldots, \lambda_r\}$ be such that it is possible to complete the set $\{\lambda_1, \ldots, \lambda_r\}$ to a set that generates all $k$-partitions of $n$. Then if $\{\mu_{1}, \ldots, \mu_{l}\}$ generates all $k$-partitions of $(n- \sum_{j=1}^r \lambda_j)$, $\{\lambda_1, \ldots, \lambda_r, \mu_{1}, \ldots, \mu_{l}\}$ is such a completion
\end{observation}

\begin{proof}

Suppose $\{\mu_{1}, \ldots, \mu_{l}\}$ generates all $k$-partitions of $(n- \sum_{j=1}^r \lambda_j)$. 

 Consider a partition $\gamma  \vdash n$ with $|\gamma| \leq k$. Our assumption on the $\lambda_j$ tell us we can find a disjoint set decomposition $\cup_{i=1}^{|\gamma|}  S_i = \{1,\ldots, r\}$ and  $\beta_i, \ldots, \beta_{|\gamma|}$ such that
\[
	\gamma_i = \sum_{j \in S_i} \lambda_j + \beta_i \qquad \forall i \in \{1, \ldots, |\gamma|\}. 
\]
We must have $\beta \vdash (n- \sum_{j=1}^r \lambda_j)$. Thus $\mu$ generates $\beta$, and we have a disjoint set decomposition $\cup_{i=1}^k T_i = \{1, \ldots, l\}$ for which 
\[
	\gamma_i = \sum_{j \in S_i} \lambda_j + \sum_{j \in T_i} \mu_j \qquad \forall i \in \{1, \ldots, |\gamma|\},
\]
or equivalently $\{\lambda_1, \ldots, \lambda_r, \mu_{1}, \ldots, \mu_{l}\}$ generates $\gamma$.

\end{proof}

For any $n$ Observation \ref{obs:average} tells us that $\{\lceil n/k \rceil\}$ can always be completed to set that generates all $k$-partitions of $n$. In fact $\lceil n/k \rceil$ is the largest value for which this is true, as adding any larger value would preclude being able to generate partitions with largest value $\lceil n/k \rceil$ while having the generator remain a partition.

Our algorithm works greedily.  It first makes $\mu_1$ the largest number such that it is always possible to complete $\{\mu_1\}$ to a generating set for all $k$-partitions of $n$. Observation \ref{obs:completion} tells us that to complete our generating set we then need only find a generating set for all $k$-partitions of $n - \mu_1$. We make $\mu_2$ the largest number such that it would be possible to complete $\{\mu_2\}$ to a generating set for all $k$-partitions of $n - \mu_1$. To complete $\{\mu_2\}$ to our desired generating set we need only find a generating set for the all $k$-partitions of $n - (\mu_1 + \mu_2)$. We continue this until we consider $n =1$ at which point the largest number such that we could complete the generating set itself completes the generating set, and we then have that the union of all of the $\{\mu_i\}$ must generate all $k$-partitions of $n$.
\vspace{.4in}

\begin{table}[h]
\begin{center}
$n$ \vspace{.3cm} \\
$k$ \begin{tabular}{c|cccccccccccccccccccc}

   & 1 & 2 & 3 & 4 & 5 & 6 & 7 & 8 & 9 & 10 & 11 & 12 & 13 & 14 & 15 & 16 & 17 & 18 & 19 & 20 \\
\midrule
 1 & 1 & 1 & 1 & 1 & 1 & 1 & 1 & 1 & 1 & 1 & 1 & 1 & 1 & 1 & 1 & 1 & 1 & 1 & 1 & 1 \\
 2 & 1 & 2 & 2 & 3 & 3 & 3 & 3 & 4 & 4 & 4 & 4 & 4 & 4 & 4 & 4 & 5 & 5 & 5 & 5 & 5 \\
 3 & 1 & 2 & 3 & 3 & 4 & 4 & 4 & 5 & 5 & 5 & 5 & 6 & 6 & 6 & 6 & 6 & 6 & 7 & 7 & 7 \\
 4 & 1 & 2 & 3 & 4 & 4 & 5 & 5 & 6 & 6 & 6 & 7 & 7 & 7 & 7 & 8 & 8 & 8 & 8 & 8 & 9 \\
 5 & 1 & 2 & 3 & 4 & 5 & 5 & 6 & 6 & 7 & 7 & 7 & 8 & 8 & 8 & 9 & 9 & 9 & 9 & 10 & 10 \\
 6 & 1 & 2 & 3 & 4 & 5 & 6 & 6 & 7 & 7 & 8 & 8 & 9 & 9 & 9 & 10 & 10 & 10 & 11 & 11 & 11 \\
 7 & 1 & 2 & 3 & 4 & 5 & 6 & 7 & 7 & 8 & 8 & 9 & 9 & 10 & 10 & 10 & 11 & 11 & 11 & 12 & 12 \\
 8 & 1 & 2 & 3 & 4 & 5 & 6 & 7 & 8 & 8 & 9 & 9 & 10 & 10 & 11 & 11 & 12 & 12 & 12 & 13 & 13 \\
 9 & 1 & 2 & 3 & 4 & 5 & 6 & 7 & 8 & 9 & 9 & 10 & 10 & 11 & 11 & 12 & 12 & 13 & 13 & 13 & 14 \\
 10 & 1 & 2 & 3 & 4 & 5 & 6 & 7 & 8 & 9 & 10 & 10 & 11 & 11 & 12 & 12 & 13 & 13 & 14 & 14 & 15 \\
\end{tabular}
\caption{The size of the generating partition $\mu$ obtained from our solution for a range of $n$ and $k$. Since we show our solution is optimal this is the minimum size needed to generate all $\lambda \vdash n$ with $|\lambda| \leq k$.}
\label{tab:sequences}
\end{center}
\end{table}

\end{section}

\begin{section}{Solution is optimal}

To prove the algorithm produces an optimal solution we will use something akin to a converse of Observation \ref{obs:completion}. 

\begin{lemma}
\label{lem:remove_max}
Let $\lambda = (\lambda_1 \geq \lambda_2 \geq \cdots \geq \lambda_{l})$ generate all $k$-partitions of $\sum_{j=1}^l \lambda_j$. Then for all $m \in \{2,\ldots,l\}$, we have
\[
	\{\lambda_m, \ldots, \lambda_{l}\} \text{ generates all $k$-partitions of } \sum_{j=m}^l \lambda_j.
\]
\end{lemma}

 To obtain this result we will use the following lemma. 

\begin{lemma}
\label{lem:greedy}
Suppose $\lambda$ generates all $k$-partitions of $n$. Then for any $\gamma \vdash n$ with $|\gamma| \leq k$, it is possible to generate $\gamma$ greedily with $\lambda$ as follows. Place the largest $\lambda_j$, $\lambda_1$, in the sum generating the largest $\gamma_i$, $\gamma_1$. For each $j \in \{2,\ldots, |\lambda|\}$ iteratively place $\lambda_j$ in the sum generating the largest $\gamma_i$ for which $\lambda_j$ can still fit in the sum. 

For example, by our algorithm (3,2,2,1,1) generates all 3-partitions of 9. To generate (4,3,2) we can use
\[
	(4,3,2) = (2+2, 3, 1+1),
\]
but this lemma tells us we are also guaranteed to be able to use
\[
	(3+1, 2 + 1, 2).
\]
\end{lemma}

\begin{proof}[Proof of Lemma \ref{lem:greedy}]

Let $\gamma \vdash n$ with $|\gamma| \leq k$. Suppose towards a contradiction that there is a point in our greedy generation of $\gamma$ at which we are trying to add a value $\lambda_{j^*}$, but it does not fit in any of the sums. 

There is a disjoint set decomposition $\cup_{i =1}^{|\gamma|} S_i = \{1, \ldots, j^*-1\}$ such that the state of our generation when we are trying to add $\lambda_{j^*}$ looks like
\[
	\gamma_i = \left(\sum_{j \in S_i} \lambda_j \right)+ \alpha_i \qquad \forall i \in \{1, \ldots, |\gamma|\},
\]
where $\alpha_i$ represents the amount of $\gamma_i$ we have yet to generate. We must have 
\[
	\alpha_i < \lambda_{j^*} \qquad \forall i \in \{1, \ldots, |\gamma|\}.
\]

The key to this argument is that since $\lambda$ can generate all $k$-partitions of $n$ it can generate the following partition
\begin{align*}
&\alpha_{|\gamma|} \\
&\alpha_{|\gamma|-1}\\
&\:\vdots \\
&\alpha_{2}\\
&n - \sum_{i=2}^{|\gamma|} \alpha_i.
\end{align*}

Furthermore, since each $\alpha_i$ is less than $\lambda_{j*}$, when we generate this partition we cannot use any of the values greater than or equal to $\lambda_{j^*}$ in the sums that generate the $\alpha_i$. Thus $\lambda_{j^*}$ and all of the values in $\lambda_j$ such that $j \in \cup_{i=1}^{|\gamma|} S_i$ must be used in the generating sum for $n - \sum_{i=2}^{|\gamma|} \alpha_i$.
Thus
\[
	n - \sum_{i=2}^{|\gamma|}  \alpha_i \geq \left(\sum_{i=1}^{|\gamma|} \sum_{j \in S_i} \lambda_j\right) + \lambda_{j^*}.
\]
On the other hand we have 
\[
	n = \left(\sum_{i=1}^{|\gamma|} \sum_{j \in S_i} \lambda_j\right) + \sum_{i=1}^{|\gamma|} \alpha_i.
\]
Thus
\[
	n - \sum_{i=2}^{|\gamma|}  \alpha_i = \left(\sum_{i=1}^{|\gamma|} \sum_{j \in S_i} \lambda_j\right)  + \alpha_1.
\]
Combining the above
\[
	\lambda_{j^*} \leq \alpha_1,
\]
a contradiction. Having arrived at a contradiction we conclude that at every step $j$ in our greedy generation we are able to add $\lambda_j$ to our generating sums. So we can introduce all the $\lambda_j$ without our generating sums ever exceeding the $\gamma_i$, which means that since introducing all of them introduces a total sum of $n$, we generate all of $\gamma$. 

\end{proof}

We are now ready to prove Lemma \ref{lem:remove_max}.

\begin{proof}[Proof of Lemma \ref{lem:remove_max}]
It clearly suffices to prove the claim for $m =2$. 

The proof is rather straightforward given Lemma $\ref{lem:greedy}$. Suppose $\lambda$ generates all $k$-partitions of $n$ with $|\lambda| = l$ and suppose $\nu$ is a partition of $(n - \lambda_1) = \sum_{j=2}^l \lambda_j$ with $|\nu| \leq k$. We construct a partition $\gamma$ of $n$ of the same size as $\nu$ by letting
\[
	\gamma_1 = \lambda_1 + \nu_1, \qquad \gamma_i = \nu_i \quad \forall i \in \{2, \ldots, |\nu|\}.
\]
Then $\gamma_1 = \lambda_1 + \nu_1$ is clearly the maximum element of $\gamma$. Thus by Lemma $\ref{lem:greedy}$ it is possible to generate $\gamma$ using $\lambda$ by first adding $\lambda_1$ to the generating sum for $\gamma_1$. Then the fact that we can fill out the rest of the $\gamma$ generating sums is precisely equivalent to we can generate $\nu$ using $(\lambda_2, \ldots, \lambda_l)$.

\end{proof}

Now to prove optimality we rely on one more lemma.

\begin{lemma}
\label{lem:sum_inequality}
Let $\mu$ be the solution obtained by the algorithm. Let $\lambda$ be a partition that also generates all $k$-partitions of $n$. Then
\[
	\sum_{i=1}^m \lambda_i \leq \sum_{i=1}^m \mu_i \qquad \forall m \in \{1,\ldots, \min(|\mu|,|\lambda|)\}.
\]

\end{lemma}

\begin{proof}
We induct on $m$. For $m =1$ this is clear. If $\lambda_1 > \mu_1$ then $\lambda_1 > \lceil n/k \rceil$ which implies $\lambda_1$ cannot be used in generating any partitions $\gamma$ that have $\gamma_1 = \lceil n/k \rceil$. Thus $\lambda$ would not be feasible.

 Now suppose the result holds for $m-1$. Clearly 
\[
	\sum_{i=1}^m \lambda_i = \left(\sum_{i=1}^{m-1} \lambda_i\right) + \lambda_m.
\]
We know something about $\lambda_m$. By Lemma \ref{lem:remove_max}, $\lambda_m$ is the largest value in a partition that generates all $k$-partitions of $n-\sum_{i=1}^{m-1} \lambda_i$. Thus $\lambda_m$ is less than or equal to the largest such a value could be, $\left\lceil (n-\sum_{i=1}^{m-1} \lambda_i)/k\right\rceil$. Thus 
\begin{align*}
\sum_{i=1}^m \lambda_i & = \left(\sum_{i=1}^{m-1} \lambda_i\right) + \lambda_m \\
& \leq \left(\sum_{i=1}^{m-1} \lambda_i\right) + \left\lceil \left(n-\sum_{i=1}^{m-1} \lambda_i\right)\bigg/k\right\rceil \\
& =  \left\lceil \left(\sum_{i=1}^{m-1} \lambda_i\right)+ \left(n-\sum_{i=1}^{m-1} \lambda_i\right) \bigg/k\right\rceil \\
& =  \left\lceil \frac n k + \pfrac{k-1}k \sum_{i=1}^{m-1} \lambda_i \right\rceil  \\
& \leq \left\lceil \frac n k + \pfrac{k-1}k \sum_{i=1}^{m-1} \mu_i \right\rceil \\
& =  \left(\sum_{i=1}^{m-1} \mu_i \right)+ \left\lceil \left(n-\sum_{i=1}^{m-1} \mu_i\right)\bigg/k\right\rceil \\
& =  \left(\sum_{i=1}^{m-1}\mu_i \right) + \mu_m \\
& =  \sum_{i=1}^{m} \mu_i.
\end{align*}
\end{proof}

 We are now ready to prove the main theorem.

\begin{theorem}
The algorithm presented produces an optimal solution.
\end{theorem}
\begin{proof}
Let $\mu$ be the solution obtained from the algorithm. Suppose towards a contradiction that there exists a $\lambda^*$, with $|\lambda^*| < |\mu|$, that generates all $k$-partitions of $n$.

By Lemma \ref{lem:sum_inequality}
\[
\sum_{i=1}^{|\lambda^*|} \lambda^*_i \leq \sum_{i=1}^{|\lambda^*|} \mu_i < n.
\]
a contradiction to $\lambda^*$ partitioning $n$. Having arrived at a contradiction we conclude the solution generated by the algorithm is optimal.
\end{proof}

\end{section}

\begin{section}{Application to demand fulfillment optimization}

There are many optimization problems that have essentially the following form with some auxiliary structure. We have a set of customers each having some demand and we have a set of fulfillers capable of filling this demand. Oftentimes there is an assumption that any customer's demand must be filled entirely by a single fulfiller. However, this assumption can be relaxed. That is we can allow for a customer's demand to be filled by multiple fulfillers each contributing a portion of the demand. Since a customer's demand is split over fulfillers we will call this the split version of our problem. 

Say customer $i$ has demand $n_i$ and there are $k$ fulfillers. When the demands and the amount of demand a fulfiller could contribute to any customer are integral a feasible solution to the split problem induces a bunch of $\lambda_i \vdash n_i$ of size less than or equal to $k$. Namely 
\[
	\lambda_i =\text{(part of $n_i$ serviced by fulfiller 1, $\ldots$ , part of $n_i$ serviced by fulfiller $k$)}.
\]

This observation opens the door to reducing the split version of the problem to the non-split version. For problems for which the split version is considerably harder than the non-split version this is desirable. One such problem is the Capacitated Vehicle Routing Problem (CVRP). 

The idea for constructing an instance of the non-split version of such a problem that at least well-approximates the split problem was first introduced by Chen et al. \cite{chen_2017}. The construction is as follows. For each $i$ we split the demand $n_i$ into some partition $\mu_i \vdash n_i$ and split customer $i$ into $|\mu_i|$ copies of itself each having demand a different element of $\mu_i$. We then solve the non-split problem on the resulting set of customers. We see that we can easily recover a solution to the split problem by having each fulfiller fill an amount of customer $i$'s original demand given by the total demand it filled on the copies of customer $i$ in the non-split problem.

Now observe that in order for this construction to be a proper reduction, i.e. in order for the non-split problem to have the same optimal objective as the split problem, we must have that for all $i$, $\mu_i$ generates the $k$-partition of $n_i$ induced by the optimal solution to the split problem. Thus it would be sufficient for $\mu_i$ to generate all $k$-partitions of $n_i$. We obviously know the partition composed of $n_i$ (1)s would do this, but this would mean copying every customer $n_i$ times, making the non-split problem instance very large. Clearly this is where our result comes in. By using the minimum size $\mu_i$ such that $\mu_i$ generates all $k$-partitions of $n$ we guarantee the equality of the objective values of the split and non-split problems without making the problem size too large. In fact, we know for sure that the number of customers in the reduced non-split problem is bounded above by the sum, over all split-problem customers $i$, of $(\lceil\log_{k/(k-1)}(n_i) \rceil + 1)$.

Furthermore, if the split problem structure is such that for all $i$, for all $\lambda \vdash n_i$ with $|\lambda| \leq k$, there is a split problem instance for which the partition for customer $i$ induced by the optimal solution is $\lambda$, then having each $\mu_i$ generates all $k$-partitions of $n_i$ is not only sufficient but necessary for the objective values to be equal. Therefore in this case we have that using our solution to choose the $\mu_i$ minimizes the size of the reduced problem.

\end{section}

\bibliography{partitions_citations.bib}

\begin{thebibliography}{1}

\bibitem{chen_2017}
P.~Chen, B.~Golden, X.~Wang, and E.~Wasil, ``A novel approach to solve the
  split delivery vehicle routing problem,'' {\em International Transactions in
  Operational Research}, vol.~24, no.~1-2, pp.~27--41, 2017.

\end{thebibliography}
\bibliographystyle{ieeetr}

\end{document}